\documentclass[12pt, reqno]{amsart}

\usepackage{amsmath, amssymb, amsfonts, amsbsy, amsthm, latexsym}
\usepackage{graphicx}
\usepackage{verbatim}
\usepackage{color,soul}
\usepackage{xy}
\usepackage{accents}

\newcommand*{\dt}[1]{%
  \accentset{\mbox{\bfseries .}}{#1}}

\newtheorem{theorem}{Theorem}[section]
\newtheorem{proposition}[theorem]{Proposition}
\newtheorem{lemma}[theorem]{Lemma}
\newtheorem{corollary}[theorem]{Corollary}

\theoremstyle{definition}
\newtheorem{definition}[theorem]{Definition}
\newtheorem{example}[theorem]{Example}

\numberwithin{equation}{section}

\numberwithin{equation}{section}

\usepackage{enumerate}

\def\0{{\bar{0}}}

\def\C{{\mathbb{C}}}

\def\R{{\mathbb{R}}}
\def\Z{{\mathbb{Z}}}

\begin{document}

\title[A sharp Balian-Low uncertainty principle for shift-invariant spaces]{A sharp Balian-Low uncertainty principle for shift-invariant spaces}

\author[D.P. Hardin]{Douglas P. Hardin}
\address{Vanderbilt University, Department of Mathematics,
Nashville, TN 37240}
\email{doug.hardin@vanderbilt.edu}
\thanks{D.~Hardin was partially supported by NSF DMS Grant 1109266, NSF DMS Grant 0934630, NSF DMS Grant 1521749, NSF DMS Grant 1516400, and
NSF DMS Grant 1412428}

\author[M.C. Northington]{Michael C. Northington V}
\address{Vanderbilt University, Department of Mathematics,
Nashville, TN 37240}
\email{michael.c.northington.v@vanderbilt.edu}
\thanks{M.~Northington was partially supported by NSF DMS Grant 0934630, NSF DMS Grant 1521749, and NSF DMS Grant 1211687}

\author[A.M. Powell]{Alexander M. Powell}
\address{Vanderbilt University, Department of Mathematics,
Nashville, TN 37240}
\email{alexander.m.powell@vanderbilt.edu}
\thanks{A.~Powell was partially supported by NSF DMS Grant 1211687 and NSF DMS Grant 1521749.}



\begin{abstract}
A sharp version of the Balian-Low theorem is proven for the generators of finitely generated shift-invariant spaces.
If generators $\{f_k\}_{k=1}^K \subset L^2(\R^d)$ are translated along a lattice to form a frame or Riesz basis for a shift-invariant space $V$, and if $V$ has extra invariance by a suitable finer lattice,
then one of the generators $f_k$ must satisfy $\int_{\R^d} |x| \thinspace |f_k(x)|^2 dx = \infty$, 
namely, $\widehat{f_k} \notin H^{1/2}(\R^d)$.  
Similar results are proven for frames of translates that are not Riesz bases without the assumption of extra lattice invariance.
The best previously existing results in the literature give a notably weaker conclusion using the Sobolev space $H^{d/2+\epsilon}(\R^d)$; 
our results provide an absolutely sharp improvement with $H^{1/2}(\R^d)$.
Our results are sharp in the sense that $H^{1/2}(\R^d)$ cannot be replaced by $H^s(\R^d)$ for any $s<1/2$.
\end{abstract}

\maketitle

\section{Introduction}

The uncertainty principle in harmonic analysis is a class of results which constrains how well-localized a function $f$ and its Fourier transform $\widehat{f}$ can be.
A classical expression of the uncertainty principle is given by the $d$-dimensional Heisenberg inequality
\begin{equation} \label{heisenberg}
\forall  f \in L^2(\R^d), \ \ \ \ \ \left( \int_{\R^d} |x|^2 |f(x)|^2 dx \right) \left( \int_{\R^d} |\xi|^2 |\widehat{f}(\xi)|^2 d\xi \right) \geq \frac{d^2}{16 \pi^2} \|f\|_{L^2(\R^d)}^4,
\end{equation}
where the Fourier transform $\widehat{f} \in L^2(\R^d)$ is defined using $\widehat{f}(\xi) = \int_{\R^d} f(x) e^{-2\pi i x \cdot \xi } dx$.
For background on this and other uncertainty principles, see \cite{FS, HJ}.  

There exist versions of the uncertainty principle which not only constrain time and frequency localization of an individual function as in \eqref{heisenberg}, 
but instead constrain the collective time and frequency localization of orthonormal bases and other structured spanning systems such as frames and Riesz bases.  
A collection $\{h_n\}_{n=1}^{\infty}$ in a Hilbert space $\mathcal{H}$ is a {\em frame} for $\mathcal{H}$ if there exist constants $0<A \leq B < \infty$
such that $$\forall h \in \mathcal{H}, \ \ \ A \|h\|_{\mathcal{H}}^2 \leq \sum_{n=1}^{\infty} | \langle h, h_n \rangle_{\mathcal{H}} |^2 \leq B \|h\|_{\mathcal{H}}^2.$$
The collection $\{h_n\}_{n=1}^{\infty}$ is a {\em Riesz basis} for $\mathcal{H}$ if it is a minimal frame for $\mathcal{H}$, i.e., $\{h_n\}_{n=1}^{\infty}$ is a frame
for $\mathcal{H}$ but $\{h_n\}_{n=1}^{\infty} \backslash \{h_N\}$ is not a frame for $\mathcal{H}$ for any $N\geq1$.  
Equivalently, $\{h_n\}_{n=1}^{\infty}$ is a Riesz basis for $\mathcal{H}$ if and only if $\{h_n\}_{n=1}^{\infty}$ is the image of an orthonormal basis under a bounded invertible operator from $\mathcal{H}$ to $\mathcal{H}$.  Every orthonormal basis is automatically a Riesz basis and a frame, but there exist frames that are not Riesz bases,
and Riesz bases that are not orthornormal bases.  See \cite{C} for background on frames and Riesz bases.

The following beautiful example of an uncertainty principle for Riesz bases was proven in \cite{GM}.   
If $\{f_n\}_{n=1}^{\infty} \subset L^2(\R^d)$ satisfies
\begin{equation}
\sup_n \left( \int_{\R^d} |x - a_n|^{2d+\epsilon} |f_n(x)|^2 dx \right) \left( \int_{\R^d} |\xi - b_n|^{2d+\epsilon} |\widehat{f_n}(\xi)|^2 d\xi \right) < \infty,
\end{equation}
for some $\epsilon>0$ and $\{(a_n, b_n)\}_{n=1}^{\infty} \subset \R^2$, then $\{f_n\}_{n=1}^{\infty}$ cannot be a Riesz basis for $L^2(\R^d)$.  Moreover,
this result is sharp in that $\epsilon$ cannot be taken to be zero, see \cite{Bour, GM}.  

There has been particular interest in uncertainty principles for bases that are endowed with an underlying group structure.  
The Balian-Low theorem for Gabor systems is a celebrated result of this type.  Given $f\in L^2(\R)$ the associated {\em Gabor system}
$\mathcal{G}(f,1,1) = \{f_{m,n}\}_{m,n \in \mathbb{Z}}$ is defined by $f_{m,n}(x) = e^{2\pi i m x} f(x - n)$.
The following nonsymmetric version of the Balian-Low theorem states that if $\mathcal{G}(f,1,1)$ is a Riesz basis for $L^2(\R)$ then $f$ must be poorly localized in either time or frequency.

\begin{theorem}[Balian-Low theorems]  \label{blt}
Let $f\in L^2(\R)$ and suppose that $\mathcal{G}(f,1,1)$ is a Riesz basis for $L^2(\R)$.
\begin{enumerate}
\item If $1<p<\infty$ and $\frac{1}{p}+\frac{1}{q}=1$, then
$$\left( \int_{\R} |x|^p |f(x)|^2 dx \right) \left( \int_{\R} |\xi|^q |\widehat{f}(\xi)|^2 d\xi \right) = \infty.$$
\item If $\widehat{f}$ is compactly supported, then 
$$\int_{\R} |x| \thinspace |f(x)|^2 d x = \infty.$$
The same result holds with the roles of $f$ and $\widehat{f}$ interchanged.
\end{enumerate}
\end{theorem}
The original Balian-Low theorem \cite{Bal81,Low} formulated 
the case $p=q=2$ in part (1) of Theorem \ref{blt} for orthonormal bases.  
The non-symmetrically weighted $(p,q)$ versions with $p\neq q$  in Theorem \ref{blt} were subsequently proven in \cite{G}.  
There are numerous extensions of the Balian-Low theorem, e.g., see the surveys \cite{BHW, CP} and articles \cite{AFK, BD, Bat88, BCGP, BCM, BCP, BCPS, DJ, GH, GHHK, HP1, HP2, J, L, NO1, NO2}.

\subsection*{Overview and main results}

In this paper we will focus on the interesting recent extensions \cite{ASW, TW} of the Balian-Low theorem to the setting of shift-invariant spaces.  Our main goal is to {\em prove sharp versions of Balian-Low type theorems in shift-invariant spaces}.  

Let us begin by recalling some notation on shift-invariant spaces.  Given $f \in L^2(\R^d)$ and $\lambda \in \R^d$ the
translation operator $T_{\lambda}:L^2(\R^d) \to L^2(\R^d)$ is defined by $T_{\lambda}f(x) = f(x-\lambda)$.
\begin{definition} 
Let $\Lambda, \Gamma$ be lattices in $\R^d$ with $\Lambda \subset \Gamma$.  Fix a $K$-tuple 
$F=(f_1, \cdots, f_K)$
where each $f_k \in L^2(\R^d)$.  With slight abuse of notation, we denote this by $F=\{f_k\}_{k=1}^K \subset L^2(\R^d)$. Assume at least one $f_k$ satisfies $||f_k||_2\neq 0$, i.e., $F$ is nontrivial.
\begin{enumerate}
\item $\mathcal{T}^{\Lambda} (F)$ denotes the system of translations $\{T_{\lambda} f : f \in F \hbox{ and } \lambda \in \Lambda\} $ viewed as a multiset.
\item $V^{\Lambda}(F)=V^\Lambda(f_1,\cdots,f_K)$ denotes the closed linear span of $\mathcal{T}^{\Lambda} (F)$ in $L^2(\R^d)$.  The space $V^\Lambda(F)$ is said to be a finitely generated shift-invariant space generated by $F$. We shall call the elements of $F$ generators of $V^\Lambda(F)$.
\item If $F =\{f \}$ consists of a single function, then $V^{\Lambda}(F) = V^{\Lambda}(f)$ is said to be a singly generated (or principal) shift-invariant space with 
{generator} $f$.
\item The minimal number of generators $\rho(F,\Lambda)$ of the space $V^{\Lambda}(F)$ is defined by 
\[\rho(F,\Lambda)=\min\{N\in \mathbb{N}: \exists\text{ an $N$-tuple } G=\{g_n\}_{n=1}^N \text{ such that } V^\Lambda(G)=V^\Lambda(F)\}.\]
\item $V^{\Lambda}(F)$ is said to be {\em $\Gamma$-invariant} if $f \in V^{\Lambda}(F)$ implies that $T_{\gamma} \thinspace f \in V^{\Lambda}(F)$ 
for all $\gamma \in \Gamma$. 
\item  $V^{\Lambda}(F)$ is said to be {\em translation invariant} if $f \in V^{\Lambda}(F)$ implies that $T_t \thinspace f \in V^{\Lambda}(F)$ for all $t\in \mathbb{R}^d$.
\end{enumerate}
\end{definition}

In contrast with the Balian-Low theorem for Gabor systems, it is possible for $f \in L^2(\R^d)$ to be well-localized in both time and frequency 
and for the system of shifts $\mathcal{T}^{\Lambda}(f)$ to be a Riesz basis for $V^{\Lambda}(f)$.  
For example, if $f\in C^{\infty}(\R)$ is compactly supported in $[-1/2,1/2]$, then $\mathcal{T}^{\mathbb{Z}}(f)$ is an orthonormal
basis for $V^{\mathbb{Z}}(f)$.  In view of this, a Balian-Low type theorem will not hold for shift-invariant spaces unless extra assumptions on the space are considered.

Our first main result is the following.  This result resolves a question posed in \cite{TW} concerning the sharp scale of Sobolev spaces needed for Balian-Low type theorems in shift-invariant spaces.

\begin{theorem} \label{sharp-blt} 
Fix lattices $\Lambda, \Gamma \subset \R^d$ with $\Lambda \subset \Gamma$ and $[\Gamma : \Lambda ]>1$.   Suppose that  $F=\{f_k\}_{k=1}^K \subset L^2(\R^d)$ 
is nontrivial and that
$\mathcal{T}^{\Lambda}(F)$ is a frame for $V^{\Lambda}(F)$.
If $[\Gamma: \Lambda]$ is not a divisor of $\rho(F,\Lambda)$ and $V^{\Lambda}(F)$ is $\Gamma$-invariant, then 
$$\exists \thinspace 1 \leq k \leq K \ \hbox{ such that } \ \ \int_{\R^d} |x| \thinspace |f_k(x)|^2 dx = \infty.$$
In other words, at least one of the generators satisfies $\widehat{f_k} \notin H^{1/2}(\R^d)$.
\end{theorem}

Here, $[\Gamma : \Lambda ]$ denotes the index of the lattice $\Lambda$ in $\Gamma$, see Section \ref{lattice-sec}.  
For singly generated shift-invariant spaces, Theorem \ref{sharp-blt} takes the following form.
\begin{corollary}
Fix lattices $\Lambda, \Gamma \subset \R^d$ with $\Lambda \subset \Gamma$ and $[\Gamma : \Lambda ]>1$.
Suppose $f \in L^2(\R^d)$, $\|f\|_2 \neq 0$, and $\mathcal{T}^{\Lambda}(f)$ forms a frame for $V^\Lambda(f)$.  If $V^{\Lambda}(f)$ is $\Gamma$-invariant, then $\int_{\R^d}  |x| \thinspace |f(x)|^2 dx = \infty$.
\end{corollary}

To put Theorem \ref{sharp-blt} in perspective, note that all previously existing results in the literature, see \cite{ASW, TW}, either give a weaker conclusion or require stronger hypotheses.
In particular, the foundational Theorem 1.2 in \cite{ASW} addresses singly generated shift-invariant spaces in dimension $d=1$ and 
gives the weaker conclusion that the generator $f\in L^2(\R)$ satisfies $\widehat{f} \notin H^{1/2+ \epsilon}(\R)$ whenever $\epsilon>0$.  
The situation is more extreme in higher dimensions $d \geq 1$, where 
Theorem 1.3 in \cite{TW} gives the weaker conclusion that at least one generator satisfies  
$\widehat{f_k} \notin H^{d/2+ \epsilon}(\R^d)$.
On the other hand, Theorem 1.2 in \cite{TW} shows if the hypothesis of ${\Gamma}$-invariance is replaced
by the notably stronger hypothesis of translation invariance, then at least one generator satisfies $\widehat{f_k} \notin H^{1/2}(\R^d)$.

Theorem \ref{sharp-blt} is sharp in the sense that $H^{1/2}(\R^d)$ cannot be replaced by $H^s(\R^d)$ when $s<1/2$.  
For example,  if $\chi_I$ is the characteristic function of the set $I=[-1/2,1/2]^d$ and 
$f(x) = \widehat{\chi_I}(x)$, then the space $V^{\mathbb{Z}^d}(f)$ is translation invariant and $\widehat{f} \in H^s(\R^d)$ for all $0< s <1/2$, cf. Proposition 1.5 in \cite{ASW} and Proposition 1.5 in \cite{TW}.

Theorem \ref{sharp-blt} is precise in the sense that it is possible for only one generator in a multiply generated system to suffer from the localization constraint $f_k \notin H^{1/2}(\R^d)$. 
In particular, we construct examples of $F=\{f_k\}_{k=1}^K$ that satisfy the hypotheses of Theorem \ref{sharp-blt},
and where $f_K \notin H^{1/2}(\R^d)$ but all other generators $f_1, \cdots, f_{K-1}$ are in $H^{1/2}(\R^d)$.  
This answers a question posed in \cite{TW} about the proportion of generators with good localization.
See Examples  \ref{example-2gen-Rd} and \ref{example-Ngen-R} in Section \ref{examples-sec} for details.

Note that Theorem \ref{sharp-blt} does not contain the compact support hypothesis that is needed in part (2) of Theorem \ref{blt}.
For perspective, Theorem \ref{blt} requires the condition that $\mathcal{G}(f,1,1)$ is a Riesz basis for the entire space $L^2(\R)$, whereas
Theorem \ref{sharp-blt} only requires the weaker assumption that $\{f(x-n): n \in \mathbb{Z} \}$ is a frame for its closed linear span $V^{\Z}(f)$ in $L^2(\R)$.  
Moreover, it is known that $\{f(x-n): n \in \mathbb{Z} \}$ cannot be a frame for the entire space $L^2(\R)$, e.g., see the literature on Gabor density theorems, \cite{H}.

Our second main result is the following.  

\begin{theorem} \label{frame-not-riesz} 
Fix a lattice $\Lambda \subset \R^d$.   Suppose that  $F=\{f_k\}_{k=1}^K \subset L^2(\R^d)$ is nontrivial and that
$\mathcal{T}^{\Lambda}(F)$ is a frame for $V^{\Lambda}(F)$, but is not a Riesz basis for $V^{\Lambda}(F)$.  If $K=\rho(F,\Lambda)$ then 
$$\exists \thinspace 1 \leq k \leq K \ \hbox{ such that } \ \ \int_{\R^d} |x| \thinspace |f_k(x)|^2 dx = \infty.$$
\end{theorem}

For singly generated shift-invariant spaces, Theorem \ref{frame-not-riesz} takes the following form.
\begin{corollary} \label{frame-not-riesz-cor}
Fix a lattice $\Lambda \subset \R^d$.  Suppose $f \in L^2(\R^d)$ with $\|f\|_2\neq 0$.  
If $\mathcal{T}^{\Lambda}(f)$ is a frame for $V^{\Lambda}(f)$, but is not a Riesz basis for $V^{\Lambda}(f)$, then 
$\int_{\R^d} |x| \thinspace |f(x)|^2 dx = \infty,$ i.e., $\widehat{f} \not\in H^{1/2}(\R^d)$.
\end{corollary}

Corollary \ref{frame-not-riesz-cor} stated with the weaker conclusion $\widehat{f} \not\in H^{d/2+\epsilon}(\R^d)$ (or more generally that $f$ is not integrable) 
may be considered folklore \cite{Groch}.  
The conclusion of Corollary \ref{frame-not-riesz-cor} with the condition $\widehat{f} \not\in H^{1/2}(\R^d)$ provides a significant and sharp improvement of this.
 
Theorem \ref{frame-not-riesz} is closely related to the work in \cite{GH}.
Note that, unlike Theorem \ref{sharp-blt}, Theorem \ref{frame-not-riesz} does not require an extra lattice invariance assumption for $V^{\Lambda}(F)$.  This result is sharp, as can be seen by considering $V^{\mathbb{Z}^d}(f)$ with $f(x)=\widehat{\chi_J}(x)$ and $J=[0,1/2]^d$, cf. \eqref{frame-condition}.  Moreover, Example \ref{example2-2gen-Rd} shows that it is possible for only a single
generator in Theorem \ref{frame-not-riesz} to have poor localization.

The remainder of the paper is organized as follows.  Section \ref{riesz-lattice-sec} contains background on lattices and shift-invariant spaces; Section \ref{fourier-sob-sec} contains background on Fourier coefficients and Sobolev spaces.
Section \ref{main-proofs-big-section} contains the proofs of our two main results, Theorems \ref{sharp-blt} and \ref{frame-not-riesz}.  
In particular, Section \ref{sob-gram-sec} proves a necessary Sobolev-type embedding for bracket products, Section \ref{char-Hhalf-sec} proves a crucial rank property of $H^{1/2}$-valued matrices, and Section \ref{combining-sec} combines the various preparatory results to prove Theorem \ref{sharp-blt} and Theorem \ref{frame-not-riesz}.
Section \ref{examples-sec} provides examples related to the main theorems.  The Appendix includes the proof of a background lemma concerning Sobolev spaces on the torus.

\section{Shift-invariant spaces: Riesz bases, frames, extra invariance} \label{riesz-lattice-sec}

In this section we recall necessary background and notation on lattices and shift-invariant spaces.

\subsection{Lattices} \label{lattice-sec}
A set $\Gamma \subset \R^d$ is a (full-rank) lattice if there exists a $d\times d$ nonsingular matrix $A$ such that $\Gamma = A(\mathbb{Z}^d)$.
Equivalently, if the columns of $A$ are denoted by $\{a_j\}_{j=1}^d$, then $\Gamma = \{ \sum_{j=1}^d z_j a_j : z_j \in \mathbb{Z} \}$.  In other words, $\{a_j\}_{j=1}^d$
is a basis (over $\mathbb{Z}$) for $\Gamma$.  
The dual lattice associated to $\Gamma$ is defined as $\Gamma^* = \{ \xi \in \R^d : \forall x \in \Gamma, e^{2\pi i x \cdot \xi} =1 \}$.  
In terms of the matrix $A$, the dual lattice can equivalently be defined as $\Gamma^* =(A^*)^{-1}(\mathbb{Z}^d)$.

Let $\sim$ be the equivalence relation on $\R^d$ defined by $x \sim y \iff x -y \in \Gamma$.   
We shall say that a set $S \subset \R^d$ is a fundamental domain of $\Gamma$ if $S$ contains precisely one representative of every equivalence class for the relation $\sim$.
Define ${M}_{\Gamma} \subset \R^d$ by ${M}_{\Gamma} = \{ Ax : x \in [-1/2,1/2)^d \},$ and note that $M_{\Gamma}$ is a fundamental domain
of $\Gamma$.

Given nested lattices $\Lambda \subset \Gamma$, the index of $\Lambda$ in $\Gamma$ is denoted by $\left[\Gamma:\Lambda\right]$, and is defined as the order of the quotient group $\Gamma/\Lambda$ when $\Gamma$ and $\Lambda$ are viewed as discrete subgroups of $\R^d$.  Moreover, $[\Gamma : \Lambda] >1$ if and only if the inclusion $\Lambda \subset \Gamma$ is strict, i.e., $\Lambda \subsetneq \Gamma$.


A function $f$ defined on $\R^d$ will be said to be $\Gamma$-periodic if $f(x+\gamma) = f(x)$ for all $x\in \R^d$ and $\gamma \in \Gamma$.
$L^2(\R^d/\Gamma)$ consists of the $\Gamma$-periodic square-integrable functions.  Since $M_\Gamma$ can be identified with the torus $\R^d / \Gamma$,
the space $L^2(\R^d/\Gamma)$ consists of $\Gamma$-periodic extensions to $\R^d$ of $L^2(M_\Gamma)$.

\subsection{Riesz bases, frames, and shift-invariant spaces}

The question of when a system of translates forms a frame or Riesz basis for a shift-invariance space has been well-studied.
Given $F =\{f_k\}_{k=1}^K \subset L^2(\R^d)$, we shall, with slight abuse of notation, let $F(x)$ denote the $d\times 1$ column vector whose entries are $f_k(x), 1 \leq k \leq K$.
For any $x$, $F(x)F^*(x)$ is a $K \times K$ Hermitian positive semi-definite matrix.  Given a lattice $\Lambda$, we define the $\Lambda$-Gramian of $F$ to be 
$$P_{\Lambda}(F)(x)= \sum_{\lambda \in \Lambda} F(x-\lambda) F^*(x-\lambda).$$
Note that $P_{\Lambda}(F)$  is $\Lambda$-periodic and is Hermitian positive semi-definite.  Also note that in the case when $F = \{ f \}$ is a singleton, 
$P_{\Lambda}(f) = \sum_{\lambda \in \Lambda} |f(x -\lambda)|^2$.

Given $F= \{f_k\}_{k=1}^K$, let  $\widehat{F}=\{\widehat{f_k} \}_{k=1}^K$   It is known \cite{BL, DDR}
 that $\mathcal{T}^{\Lambda}(F)$ forms a Riesz basis for $V^{\Lambda}(F)$ if and only if there exists $t \ge 1$ such that 
\begin{equation}\label{riesz-condition}
t^{-1} I \le P_{\Lambda^*}(\widehat{F})(x) \le tI \indent \text{a.e. } x \in M_{\Lambda^*}.
\end{equation}
Moreover, see \cite{B, BL},  $\mathcal{T}^{\Lambda}(F)$ forms a frame for $V^{\Lambda}(F)$ if and only if there exists $t \ge 1$ such that 
\begin{equation}\label{frame-condition}
t^{-1} P_{\Lambda^*}(\widehat{F})(x) \le (P_{\Lambda^*}(\widehat{F})(x))^2 \le tP_{\Lambda^*}(\widehat{F})(x) \indent \text{a.e. } x \in M_{\Lambda^*}.
\end{equation}

The following result addresses the minimal number of generators of shift-invariant spaces, see Proposition 4.1 in \cite{TW}.

\begin{proposition} \label{prop-min-generators}
Let $\Lambda$ be a lattice in $\R^d$, and $F \subset L^2(\R^d)$.  
The minimal number of generators of $V^\Lambda(F)$ is given by
$$\rho(F,\Lambda) = {\rm ess \thinspace sup}_{x\in \R^d} \left( \emph{rank} \left[ P_{\Lambda^*}(\widehat{F})(x) \right] \right).$$  
\end{proposition}

The following theorems address properties of shift-invariant spaces $V^{\Lambda}(F)$ that are invariant under a lattice $\Gamma$ that is larger than $\Lambda$, see Theorem 2.1 and Theorem 3.2 in \cite{TW} and similar results in \cite{ACHKM} and \cite{ACP}.   For simplicity, we state the next two results for lattices $\Lambda, \Gamma$, but both
results remain true when $\Lambda \subset \Gamma$ are closed cocompact subgroups of $\R^d$.

\begin{theorem} \label{thm-rank-formula} 
Let $\Lambda, \Gamma \subset \R^d$ be lattices with $\Lambda \subset \Gamma$.
Let $R\subset \Gamma^*$ be a collection of representatives of the quotient $\Lambda^*/ \Gamma^*$ so that  
$$P_{\Lambda^*}(\widehat{F})(x)=\sum_{k \in R}  P_{\Gamma^*}(\widehat{F})(x+k), \ a.e.\  x\in \R^d.$$  
The space $V^\Lambda(F)$ is $\Gamma$-invariant if and only if 
$${\rm rank} \left[ P_{\Lambda^*}(\widehat{F})(x) \right]= \sum_{k \in R} {\rm rank} \left[ P_{\Gamma^*}(\widehat{F})(x+k) \right], \ a.e. \ x \in \R^d.$$
\end{theorem}

\begin{theorem} \label{thm-eig-bound}
Let $\Lambda, \Gamma \subset \R^d$ be lattices with $\Lambda \subset \Gamma$.
Suppose the space $V^\Lambda(F)$ is $\Gamma$-invariant and $\mathcal{T}^\Lambda(F)$ forms a frame for $V^\Lambda(F)$.  Then $\mathcal{T}^\Gamma(F)$ also forms a frame for $V^\Lambda(F)=V^\Gamma(F)$.  That is, there exists $t \ge 1$ such that for a.e. $x \in \R^d$, 
\[ t^{-1} P_{\Gamma^*}(\widehat{F})(x) \le (P_{\Gamma^*}(\widehat{F})(x))^2 \le tP_{\Gamma^*}(\widehat{F})(x) .\]
\end{theorem}

\section{Background: Fourier coefficients and Sobolev spaces} \label{fourier-sob-sec}

In this section we collect necessary background results and notation 
on Fourier coefficients and Sobolev spaces.

\vspace{.1in}
\subsection{Fourier coefficients}
Recall that the Fourier coefficients of $f \in L^2(\R^d/\Gamma)$ are defined by
$$\forall \xi \in \Gamma^*, \ \ \ \widehat{f}(\xi) = \int_{\R^d/\Gamma} f(x) e^{-2\pi i x \cdot \xi} \ dh(x) 
= \frac{1}{|{M}_{\Gamma}|} \int_{{M}_{\Gamma}} f(x) e^{-2\pi i x \cdot \xi} \ dx,$$
where $dx$ is Lebesgue measure and $dh$ is normalized Haar measure on the compact group $\R^d/\Gamma$.
Also recall Parseval's theorem
\begin{equation}
\int_{\R^d / \Gamma} |f(x)|^2 dh(x) = \frac{1}{|M_\Gamma|} \int_{{M}_\Gamma} |f(x)|^2dx = \sum_{\xi \in \Gamma^*} |\widehat{f}(\xi)|^2,
\end{equation}
and the translation property
\begin{equation}
\forall y \in \R^d, \forall \xi \in \Gamma^*, \ \ \ \widehat{T_y f}(\xi) = \widehat{f}(\xi) e^{-2\pi i y \cdot \xi}.
\end{equation}

\vspace{.1in}
\subsection{Sobolev spaces}

Given $s>0$, the Sobolev space ${H}^s(\R^d)$ consists of all measurable functions $f$ defined on $\R^d$ such that
$\|f\|_{H^s(\R^d)} =  \left( \int_{\R^d} (1+ |\xi|^{2})^s |\widehat{f}(\xi)|^2 d \xi \right)^{1/2} < \infty.$
Equivalently, $f\in H^s(\R^d)$ if and only if $f \in L^2(\R^d)$  and
\begin{equation} \label{homog-norm}
\| f \|_{\dt{H}^{s}(\R^d)} = \left( \int_{\R^d} |\xi|^{2s} |\widehat{f}(\xi)|^2 d \xi \right)^{1/2} < \infty.
\end{equation}
Recall the following equivalent characterization of \eqref{homog-norm} when $0<s<1$, e.g., \cite{S},
\begin{equation} \label{Rd-equiv-sob-norm}
\| f\|^2_{\dt{H}^{s}(\R^d)} = C(d,s) \int_{\R^d} \int_{\R^d} \frac{|f(x+y) - f(x)|^2}{|y|^{d+2s}} dx dy.
\end{equation}

\vspace{.1in}

\subsection{Sobolev spaces of periodic functions} \label{sob-torus-sec}

We shall also need some background on Sobolev spaces of periodic functions.  

\begin{definition}[Sobolev spaces on the torus]
Let $\Gamma \subset \R^d$ be a lattice with dual lattice $\Gamma^* \subset \R^d$.  
Given $s>0$, define the {Sobolev space} ${H}^s(\R^d/\Gamma) = \{ f \in L^2(\R^d/\Gamma) : \|f\|_{\dt{H}^s(\R^d/\Gamma)} < \infty \}$,
where
$\|f\|_{\dt{H}^s(\R^d/\Gamma)} = \left(  \sum_{\xi \in \Gamma^*} |\xi|^{2s} | \widehat{f}(\xi)|^2 \right)^{1/2}.$
\end{definition}

The following proposition gives a useful equivalent characterization of $\|f\|_{\dt{H}^s(\R^d/\Gamma)}$ for $0<s<1$.  Equation \eqref{eq-sob-norm1} is a version for $H^s(\R^d/\Gamma)$ of Proposition 1.3 in \cite{BO}, and
equation \eqref{eq-sob-norm2} is an extension to $H^s(\R^d/\Gamma)$ of the equivalence on page 66 in \cite{BBM}.
We use the notation $X \asymp Y$ to indicate that there exist absolute
constants $0<C_1\leq C_2$ such that $C_1 X \leq Y \leq C_2 X$.
\begin{lemma}\label{prop-equiv-norm}
Fix $0<s<1$, let $\Gamma \subset \R^d$ be lattice, and suppose that $f \in L^2(\R^d/\Gamma)$.
Then
\begin{equation}\label{eq-sob-norm1}
\|f\|^2_{\dt{H}^s(\R^d/\Gamma)} \asymp \int_{{M}_\Gamma} \int_{{M}_\Gamma} \frac{|f(x+y)-f(x)|^2}{|y|^{d+2s}} dx dy.
\end{equation}
Moreover, if $\{a_j\}_{j=1}^d \subset \Gamma$ is a basis for $\Gamma$  then
\begin{equation} \label{eq-sob-norm2}
\|f\|^2_{\dt{H}^s(\R^d/\Gamma)} \asymp \sum_{j=1}^d \int_{[-\frac{1}{2}, \frac{1}{2})}  \int_{{M}_\Gamma}\frac{|f(x+t a_j)-f(x)|^2}{|t|^{1+2s}} dx dt. 
\end{equation}
The implicit constants in 
\eqref{eq-sob-norm1} and 
\eqref{eq-sob-norm2} depend on $s,d,\Gamma$.
\end{lemma}

The proof of Lemma \ref{prop-equiv-norm} is included in the Appendix.

\section{Proofs of the main theorems}  \label{main-proofs-big-section}

This section gives proofs of our main results, Theorems \ref{sharp-blt} and \ref{frame-not-riesz}.
We have chosen to organize the proofs into digestible sections of preparatory technical results.
In particular, in Section \ref{sob-gram-sec} we prove a necessary Sobolev embedding for bracket products,
and in Section \ref{char-Hhalf-sec} we prove a crucial lemma on the rank of $H^{1/2}$-valued matrix functions.
Finally, in Section \ref{combining-sec} we combine the preparatory results and prove Theorems \ref{sharp-blt} and \ref{frame-not-riesz}.

\subsection{A Sobolev embedding for bracket products} \label{sob-gram-sec}
Given a lattice $\Lambda \subset \R^d$ and $f,g \in L^2(\R^d)$, it will be convenient to define the {\em bracket product} of $g,h$ by
$$[f,g](x) = [f,g]_{\Lambda}(x)=\sum_{\lambda \in \Lambda} f(x-\lambda) \overline{g(x-\lambda)}.$$
For background on bracket products and their connection to shift-invariant spaces see, for example, \cite{CL, HSWW}.

\begin{lemma} \label{thm-embedding}
Let $0<s<1$.  If $g,h \in H^s(\R^d)$ and $P_{\Lambda}(g), P_{\Lambda}(h) \in L^{\infty}(\R^d /\Lambda)$ then
\begin{equation} \label{brack-norm-ineq}
\bigl\| [g,h] \bigr\|_{\dt{H}^s(\R^d / \Lambda)}^2 \le 2 C(s,d) \left( \|P_\Lambda(g)\|_{L^{\infty}(\R^d /\Lambda)} \|h\|_{\dt{H}^s(\R^d)}^2 + \|P_\Lambda(h)\|_{L^{\infty}(\R^d /\Lambda)} \|g\|_{\dt{H}^s(\R^d)}^2 \right),
\end{equation}
where $C(s,d)$ is the constant in \eqref{Rd-equiv-sob-norm}.  In particular, $[g,h]_{\Lambda} \in H^s(\R^d / \Lambda)$.  
\end{lemma}

\begin{proof} 
Note that $P_{\Lambda}(g), P_{\Lambda}(h) \in L^{\infty}(\R^d /\Lambda)$ imply that $[g,h] \in L^{\infty}(\R^d /\Lambda) \subset L^2(\R^d /\Lambda).$
So, by Lemma \ref{prop-equiv-norm}, it suffices to show that 
\[\int_{M_{\Lambda}} \int_{M_{\Lambda}} \frac{\bigl| [g,h](x+y)-[g,h](x) \bigr|^2}{|y|^{d+2s}} dy dx <\infty.\]
We have,
\begin{align}
	\int_{M_{\Lambda}} & \int_{M_{\Lambda}} \frac{\bigl|[g,h](x+y)-[g,h](x) \bigr|^2}{|y|^{d+2s}} dy dx \notag \\
	&\le\int_{M_{\Lambda}}  \int_{M_{\Lambda}} \frac{ (\sum_{\lambda \in \Lambda} |g(x+y-\lambda)\overline{h(x+y-\lambda)}- g(x-\lambda)\overline{h(x-\lambda)}|)^2}{|y|^{d+2s}} dy dx \notag \\
	&\le 2\int_{M_{\Lambda}}  \int_{M_{\Lambda}} \frac{ (\sum_{\lambda \in \Lambda} |g(x+y-\lambda)||{h(x+y-\lambda)}-h(x-\lambda)|)^2}{ |y|^{d+2s}}dy dx \label{brack-eq1}\\
	&+ 2\int_{M_{\Lambda}}  \int_{M_{\Lambda}} \frac{(\sum_{\lambda \in \Lambda}|{h(x-\lambda)}||g(x+y-\lambda)-g(x-\lambda)|)^2}{|y|^{d+2s}} dy dx. \label{brack-eq2}
\end{align}

Using \eqref{Rd-equiv-sob-norm}, the expression in \eqref{brack-eq1} can be bounded as follows,
\begin{align*}
	\int_{M_{\Lambda}}  &\int_{M_{\Lambda}} \frac{ (\sum_{\lambda \in \Lambda} |g(x+y-\lambda)||{h}	(x+y-\lambda)-{h}(x-\lambda)|)^2}{ |y|^{d+2s}}dydx \\
	&\le \int_{M_{\Lambda}}  \int_{M_{\Lambda}}\frac{ \sum_{\lambda \in \Lambda} |g(x+y-\lambda)|^2\sum_{\lambda \in \Lambda}|{h}(x+y-\lambda)-{h}(x-\lambda)|^2}{ |y|^{d+2s}} dy dx\\
	& \le \|P_\Lambda(g)\|_{L^\infty(\R^d / \Lambda)} \int_{M_{\Lambda}}  \int_{M_{\Lambda}}\frac{\sum_{\lambda \in \Lambda}|{h}(x+y-\lambda)-{h}(x-\lambda)|^2}{ |y|^{d+2s}} dy dx\\
	& \le \|P_\Lambda(g)\|_{L^\infty(\R^d / \Lambda)} \int_{\R^d}  \int_{M_{\Lambda}}\frac{|h(x+y)-h(x)|^2}{ |y|^{d+2s}} dy dx\\
	& \le C(d,s)\|P_\Lambda(g)\|_{L^\infty(\R^d / \Lambda)} \|h\|_{\dt{H}^s(\R^d)}^2.
\end{align*}
This, together with a similar bound for \eqref{brack-eq2}, gives \eqref{brack-norm-ineq}.
\end{proof}

\subsection{Rank constraints for $H^{1/2}$-valued matrices on the torus} \label{char-Hhalf-sec}

In this section, we prove the following technical lemma which is crucially needed in the proofs of our main theorems.

\begin{lemma} \label{lem-constant-rank}
Let $\Lambda \subset \R^d$ be a lattice.  
For almost every $x$, let $P(x)$ be a Hermitian positive semi-definite $n \times n$ matrix with entries $p_{i,j} \in H^{1/2}(\R^d/\Lambda)$ for all $i,j \in \{1,...,n\}$. If, for some $t>0$, $P$ satisfies the condition 
\begin{equation}
	t P(x) \le P(x)^2, \indent \text{a.e. } x \in \R^d/\Lambda, \label{eq-frame-cond}
\end{equation}
then the rank of $P(x)$ is constant a.e.
\end{lemma}

The proof of Lemma \ref{lem-constant-rank} requires two additional preparatory results, Lemmas \ref{lem-eigenvalues} and \ref{lem-char-fun}.  
We first state and prove these preparatory lemmas, and then use them to prove Lemma \ref{lem-constant-rank} at the end of this section.  

\begin{lemma} \label{lem-eigenvalues}
Let $0<s<1$.  
Let $\Lambda \subset \R^d$ be a lattice.  
For almost every $x$, let $P(x)$ be a Hermitian positive semi-definite $n \times n$ matrix with entries $p_{i,j} \in H^s(\R^d/\Lambda)$ for all $i,j \in \{1,...,n\}$. 
For a given $x \in \R^d / {\Lambda}$ such that $P(x)$ is defined, let $\lambda_1(x)\ge ...\ge \lambda_n(x)\ge 0$ denote the eigenvalues of $P(x)$.
Then, for each $1 \leq k \leq n$, the eigenvalue function $\lambda_k \in {H}^s(\R^d / \Lambda)$.  
\end{lemma}
\begin{proof}
The Courant-Fischer-Weyl min-max theorem, e.g., Corollary III.1.2 in \cite{BH}, says that 
$$\lambda_k(\xi)=\max\{ \min\{ \langle u, P(\xi)u \rangle : u \in U, |u|=1\} : \dim(U)=k\},$$
where the minimum is taken over all $k$-dimensional subspaces $U$ of $\C^d$.
Then we have 
\begin{align*}
|\lambda_k(\eta) - \lambda_k(\xi)|
  =&| \max  \{ \min\{ \langle u, P(\eta)u \rangle : u \in U, |u|=1\} : \dim(U)=k\}\\
	&-\max\{ \min\{ \langle v, P(\xi)v \rangle : v \in V, |v|=1\} : \dim(V)=k\}|.
\end{align*}
Without loss of generality, assume that $\lambda_k(\eta)\ge \lambda_k(\xi)$.  
Choose a subspace $U_0$ that realizes the maximum $\lambda_k(\eta)$.  
Then we have 
\begin{align*}
 |\lambda_k(\eta) - \lambda_k(\xi)|  &= \min\{ \langle u, P(\eta)u \rangle : u \in U_0, |u|=1\}\\
      & \indent - \max\{ \min\{ \langle v, P(\xi)v \rangle : v \in V, |v|=1\} : \dim(V)=k\} \\
	&\le  \min\{ \langle u, P(\eta)u \rangle : u \in U_0, |u|=1\} - \min\{ \langle v, P(\xi)v \rangle : v \in U_0, |v|=1\}. 
\end{align*}
Next, choose $u_0\in U_0$ with $\|u_0\|=1$ such that the minimum in the right term is achieved at $u_0$.  Then, 
\begin{align*}
 |\lambda_k(\eta)&- \lambda_k(\xi)|\le  \min\{ \langle u, P(\eta)u \rangle : u \in U_0, |u|=1\} -  \langle u_0, P(\xi)u_0 \rangle \\
          &\le  \langle u_0, P(\eta)u_0 \rangle -  \langle u_0, P(\xi)u_0 \rangle \\
	&= \langle  u_0, (P(\eta)-P(\xi))u_0 \rangle\\
	&\le \|u_0\|^2 \|P(\eta)-P(\xi)\|_{\rm op}\\
	&\le \|P(\eta)-P(\xi)\|_{\rm Frob},
\end{align*}
where the last inequality holds since the Frobenius norm of a matrix controls the spectral norm of a matrix.
Thus, by \eqref{eq-sob-norm1}, we have 
\begin{align*}
	\int_{M_{\Lambda}}\int_{M_{\Lambda}} \frac{|\lambda_k(x+y)-\lambda_k(x)|^2}{|y|^{d+2s}} dy dx
	&\leq \int_{M_{\Lambda}}\int_{M_{\Lambda}} \frac{\|P(x+y)-P(x)\|_{\rm Frob}^2}{|y|^{d+2s}} dy dx\\
	&= \sum_{i,j=1}^n \int_{M_{\Lambda}}\int_{M_{\Lambda}} \frac{|p_{i,j}(x+y)-p_{i,j}(x)|^2}{|y|^{d+2s}} dy dx\\
	&= \sum_{i,j=1}^n \|p_{i,j}\|_{\dt{H}^s(\R^d /\Lambda)}^2 <\infty.
\end{align*}
\end{proof}

It is known that if $f$ is the characteristic function of a measurable set $S\subset \R^d$ with positive finite Lebesgue measure then 
$f \notin {H}^{1/2}(\R^d)$, e.g., see \cite{BBM}.
We need the following version of this result for the Sobolev space of periodic functions ${H}^{1/2}(\R^d/\Lambda)$ where $\Lambda\subset \R^d$ is a lattice.

\begin{lemma}\label{lem-char-fun}
Let $\Lambda \subset \R^d$ be a lattice.
Suppose that $g\in H^{1/2}(\R^d / \Lambda)$ and there exists $S \subset M_\Lambda$ such that
$g(x)=0$ for a.e. $x\in S,$ and $g(x) \geq C>0$  for a.e. $x\in M_{\Lambda} \cap S^c.$
Then either $|S|=0$ or $|S|=|M_{\Lambda}|.$
\end{lemma}
\begin{proof}
Without loss of generality, we can assume that $C=1$.  
Let $\chi_E$ be the $\Lambda$-periodic extension of $\chi_S$ to $\R^d$.
Notice that for a.e. $x,y \in \R^d$,
\[ |g(x+y)-g(x)|\ge |\chi_E(x+y)-\chi_E(x)|, \]
and so, by equation \eqref{eq-sob-norm1}, $\chi_E \in H^{1/2}(\R^d/\Lambda)$ since $g\in H^{1/2}(\R^d/\Lambda)$.  
Therefore, it suffices to prove the lemma in the case that $g = \chi_E$ for some $E\subset \R^d / \Lambda$.  
 The proof is divided into two cases depending on whether $d=1$ or $d \geq 2$.\\

{\em Case 1.} We begin by addressing the case $d=1$.  In this case $\Lambda = \alpha \Z$ for some $\alpha>0$, and $M_\Lambda= [-\alpha/2,\alpha/2)$.  For the sake of contradiction, suppose there exists a set $S \subset [-\alpha/2, \alpha/2)$ with $0<|S|<\alpha$ such that $g$ is the $\Lambda$-periodic extension
of $\chi_S$ to $\R$.  

For any interval $I \subset [-\alpha/2, \alpha/2)$, we have 
\begin{align}
\frac{1}{|I|^2}\int_I \int_I & |g(x) - g(y)| dx dy = \frac{1}{|I|^2}\int_I \int_{I-x} |g(x+y)-g(x)| dy dx \notag \\
&\leq \frac{1}{|I|^2} \left(\int_I \int_{I-x} \frac{|g(x+y)-g(x)|^2}{|y|^2} dy dx\right)^{1/2} \left(\int_I \int_{I-x} |y|^2 dy dx\right)^{1/2} \notag \\
&\leq    \left(\int_I \int_{I-x} \frac{|g(x+y)-g(x)|^2}{|y|^2} dy dx\right)^{1/2}.\label{vmo-eq1} 
\end{align}
Since $g$ is the indicator function of a set, one has
\begin{equation} \label{big-y-alpha-eq}
\int_I \int_{\alpha/2 \leq |y| \leq \alpha} \frac{|g(x+y)-g(x)|^2}{|y|^2} dy dx  \leq
\int_I \int_{\alpha/2 \leq |y| \leq \alpha} \frac{1}{|\alpha/2|^2} \ dy dx
\leq
\frac{4|I|}{\alpha}.
\end{equation}
If $x\in I \subset [-\alpha/2, \alpha/2)=M_\Lambda$ then  $I-x \subset [-\alpha, \alpha]$.  
This, together with \eqref{big-y-alpha-eq}, implies that 
\begin{align} 
\int_I \int_{I-x} \frac{|g(x+y)-g(x)|^2}{|y|^2} dy dx
& \leq \int_I \int_{-\alpha}^{\alpha} \frac{|g(x+y)-g(x)|^2}{|y|^2} dy dx \notag \\
& \leq \int_I \int_{-\alpha/2}^{\alpha/2} \frac{|g(x+y)-g(x)|^2}{|y|^2} dy dx  + \frac{4|I|}{\alpha}. \label{vmo-eq2}
\end{align}
Using $g \in H^{1/2}(\R/\Lambda)$, \eqref{eq-sob-norm1},  \eqref{vmo-eq1}, \eqref{vmo-eq2},
and absolute continuity of the Lebesgue integral, it follows that
\begin{equation}\label{lim-zero}
\lim_{|I| \to 0} \frac{1}{|I|^2} \int_{I} \int_{I} |g(x) - g(y)| dx dy =0.
\end{equation}

Since $0<|S|<\alpha$, for every sufficiently small $\epsilon>0,$ there exists an interval $Q_\epsilon \subset [-\alpha/2, \alpha/2)$ 
such that $|Q_\epsilon|<\epsilon$ and $|Q_\epsilon \cap S| = |Q_\epsilon \cap S^c| = |Q_\epsilon|/2$ 
(for example, this follows from the Lebesgue differentiation theorem).
So, for every sufficiently small $\epsilon>0$,
\begin{align}
\frac{1}{|Q_\epsilon|^2} \int_{Q_\epsilon} \int_{Q_\epsilon} |g(x) - g(y)| dx dy 
& \geq \frac{1}{|Q_\epsilon|^2} \int_{Q_\epsilon \cap S} \int_{Q_\epsilon \cap S^c} |g(x) - g(y)| dx dy \notag \\
&= \frac{1}{|Q_\epsilon|^2} \int_{Q_\epsilon \cap S} \int_{Q_\epsilon \cap S^c} 1 \ dx dy \notag \\
& = \frac{|Q_\epsilon \cap S| \thinspace |Q_\epsilon \cap S^c|}{|Q_\epsilon|^2} = 1/4. \label{vmo-eq3}
\end{align}
On the other hand, by \eqref{lim-zero}, 
\begin{equation}\label{lim-zero-eq2}
\lim_{\epsilon \to 0} \frac{1}{|Q_\epsilon|^2} \int_{Q_\epsilon} \int_{Q_\epsilon} |g(x) - g(y)| dx dy =0.
\end{equation}
Since  \eqref{vmo-eq3} and \eqref{lim-zero-eq2} form a contradiction, it follows that either $|S|=0$ or $|S|=\alpha=|M_{\Lambda}|$.\\

\vspace{.1in}
{\em Case 2.} Next, we address the case $d\ge 2$.  Suppose that $S\subset M_{\Lambda}$ and that $g$ is the $\Lambda$-periodic extension of $\chi_S$ to $\R^d$.
We will show that either $|S|=0$ or $|S|=|M_\Lambda|$.

Let $\{a_j\}_{j=1}^d \subset \Lambda$ be a basis for $\Lambda$.  
Recall that $M_\Lambda = \{ \sum_{j=1}^d t_j a_j : -1/2 \leq t_j \leq 1/2 \}$.  For each fixed $x\in M_{\Lambda}$ and $1 \leq k \leq d$, define for $t \in [-1/2, 1/2)$
$$\psi_{x,k}(t) = g(x + t a_k).$$  
Note that for a.e. $x\in M_\Lambda$, $\psi_{x,k}$ is 1-periodic and $\psi_{x,k} \in L^2(\R / \mathbb{Z})$.  Also, by \eqref{eq-sob-norm2}
\begin{align*}
\int_{M_\Lambda} \| \psi_{x,k} \|^2_{\dt{H}^{1/2}(\R / \mathbb{Z})} d x
&\asymp \int_{M_\Lambda}\int_{-1/2}^{1/2} \int_{-1/2}^{1/2} \frac{|\psi_{x,k}(s+t) - \psi_{x,k}(s)|^2}{|t|^{2}} \ ds \thinspace dt \thinspace dx\\
&=\int_{M_\Lambda}\int_{-1/2}^{1/2} \int_{-1/2}^{1/2} \frac{|g(x +(s+t)a_k) - g(x+sa_k)|^2}{|t|^{2}} \ ds \thinspace dt \thinspace dx\\
&=\int_{M_\Lambda} \int_{-1/2}^{1/2} \frac{|g(y+ta_k) - g(y)|^2}{|t|^{2}} \ dt \thinspace dy <\infty.
\end{align*}
Thus, for each $1 \leq k \leq d$ and a.e. $x\in M_\Lambda$, we have $\psi_{x,k} \in H^{1/2}(\R / \mathbb{Z})$.  However, since $g(x) \in \{ 0, 1 \}$ a.e., we also have
that for each $1\leq k \leq d$ and almost every $x \in M_\Lambda$
\begin{equation}
 \psi_{x,k}(t) \in \{ 0 , 1\}, \ \ \hbox{ for } \  a.e. \ t \in \R.
\end{equation}
It follows from Case 1 that for each $1\leq k \leq d$ and almost every $x \in M_\Lambda$
\begin{equation} \label{g-zero-or-one}
g(x+ta_k) = 0 \hbox{ for } a.e. \ t \in \R, \ \ \  \hbox{ or } \ \ \ g(x+ta_k)(t) = 1 \hbox{ for } a.e. \ t \in \R.
\end{equation}

To complete the proof it now suffices  to show that $g(x) = g(y)$ for a.e. $x,y \in M_\Lambda$.  For this, it suffices to show that
$g(\sum_{j=1}^d t_j a_j) = g(\sum_{j=1}^d s_j a_j)$ for a.e. $t = (t_1, \cdots, t_d) \in [-1/2,1/2)^d$ and $s = (s_1, \cdots, s_d) \in [-1/2,1/2)^d$.  
Similarly to Lemma 2 in \cite{BLMN}, and using \eqref{g-zero-or-one}, one has
\begin{align*}
	\int_{[-1/2,1/2)^d}&\int_{[-1/2,1/2)^d}  |g(\sum_{j=1}^d t_j a_j )-g(\sum_{j=1}^d s_j a_j )| dt ds \\
	& \leq \int_{[-1/2,1/2)^d}\int_{[-1/2,1/2)^d} |g(\sum_{j=1}^d t_j a_j )-g(s_1a_1 +\sum_{j=2}^d t_j a_j )| ds dt \\
	&+\int_{[-1/2,1/2)^d}\int_{[-1/2,1/2)^d} |g(s_1a_1+\sum_{j=2}^d t_j a_j )-g(\sum_{j=1}^2 s_ja_j +\sum_{j=3}^d t_j a_j )| ds dt \\
	& \ \vdots \\
	&+\int_{[-1/2,1/2)^d}\int_{[-1/2,1/2)^d} |g(\sum_{j=1}^{d-1} s_j a_j + t_d a_d)-g(\sum_{j=1}^{d} s_ja_j )| ds dt \\
	& = 0.
\end{align*}
Thus, $g(x)=g(y)$ for almost every $x,y \in M_{\Lambda}$.  
\end{proof}

For perspective, the hypothesis of Lemma \ref{lem-char-fun} implies the condition $C|g(x)| \leq |g(x)|^2$ for a.e. $x\in \R^d$, 
which may be viewed as a scalar version of the matrix-valued hypothesis \eqref{eq-frame-cond}.
In particular, Lemma \ref{lem-constant-rank} may be thought of as a matrix-valued generalization of Lemma \ref{lem-char-fun}.  
We are now ready to prove Lemma \ref{lem-constant-rank}.  \\

\noindent {\em Proof of Lemma \ref{lem-constant-rank}.}
Lemma \ref{lem-eigenvalues} shows that the eigenvalue functions $\lambda_k$ of $P$ are in $H^{1/2}(\R^d/\Lambda)$.   Condition (\ref{eq-frame-cond}) implies that for almost every $x\in \R^d/\Lambda$, $\lambda_k(x)=0$ or $\lambda_k(x)\ge t>0$.  From Lemma \ref{lem-char-fun} we have that $\lambda_k$ is either zero almost everywhere or positive almost everywhere.  Therefore, the rank of $P(x)$ is constant almost everywhere.  
\qed \\

\subsection{Combining everything: proofs of the main theorems.} \label{combining-sec}
In this section, we combine all of our preparatory results and prove Theorems \ref{sharp-blt} and \ref{frame-not-riesz}.\\

\noindent {\em Proof of Theorem \ref{sharp-blt}.}
Assume, for the sake of contradiction, that $\widehat{f_k} \in H^{1/2}(\R^d)$ for all $1\le k \le K$.  

{\em Step I.}  Note that the periodizations 
$P_{\Lambda^*}(\widehat{f_k})\in L^{\infty}(\R^d/\Lambda^*)$ and $P_{\Gamma^*}(\widehat{f_k}) \in L^{\infty}(\R^d/ \Gamma^*)$ for each $1 \leq k \leq K$.
To see this, 
let $e_k$ be the $k$th canonical basis vector for $\R^K$, and use \eqref{frame-condition} to obtain
$$\left( [ \widehat{f_k}, \widehat{f_k} ]_{\Lambda^*} \right)^2 
\leq \sum_{m=1}^K \left| [ \widehat{f_m}, \widehat{f_k} ]_{\Lambda^*} \right|^2 = \langle Pe_{k}, Pe_{k} \rangle = \langle P^2e_{k}, e_{k} \rangle
\leq \langle t Pe_{k}, e_{k} \rangle = t [ \widehat{f_k}, \widehat{f_k} ]_{\Lambda^*}.$$
Recalling that $P_{\Lambda^*}(\widehat{f_k}) = [ \widehat{f_k}, \widehat{f_k} ]_{\Lambda^*}$, it follows that 
$|P_{\Lambda^*}(\widehat{f_k})(x)| \leq t$ for a.e. $x\in \R^d$, so that $P_{\Lambda^*}(\widehat{f_k}) \in L^{\infty}(\R^d/\Lambda^*)$.  Similar reasoning, together with
Theorem \ref{thm-eig-bound}, shows that $P_{\Gamma^*}(\widehat{f_k}) \in L^{\infty}(\R^d/\Gamma^*)$.

{\em Step II.}  Lemma \ref{thm-embedding} implies that the bracket products satisfy $[ \widehat{f_m}, \widehat{f_n} ]_{\Lambda^*} \in H^{1/2}(\R^d / \Lambda^*)$ 
and $[ \widehat{f_m}, \widehat{f_n} ]_{\Gamma^*} \in H^{1/2}(\R^d / \Gamma^*)$  for all $1 \leq m,n \leq K$.  
So, all entries of the Gramian $P_{\Lambda^*}(\widehat{F})$ are in $H^{1/2}(\R^d/\Lambda^*)$, and
all entries of $P_{\Gamma^*}(\widehat{F})$ are in $H^{1/2}(\R^d/\Gamma^*)$.
Combining Lemma \ref{lem-constant-rank}, Theorem \ref{thm-eig-bound}, and equation \eqref{frame-condition} shows that the rank of both $P_{\Lambda^*}(\widehat{F})(x)$ and $P_{\Gamma^*}(\widehat{F})(x)$ are constant almost everywhere.

{\em Step III.}
By Proposition \ref{prop-min-generators}, $\text{rank}[P_{\Gamma^*}(\widehat{F})(x)]=\rho(F,\Gamma)$ a.e., 
and $\text{rank}[P_{\Lambda^*}(\widehat{F})(x)]=\rho(F,\Lambda)$ a.e.  Then, from Theorem \ref{thm-rank-formula}, we have 
$$\rho(F,\Lambda)=\text{rank} [P_{\Lambda^*}(\widehat{F})(x)]= \sum_{k \in R}\text{rank} [P_{\Gamma^*}(\widehat{F})(x+k)]= \rho(F,\Gamma) {\rm card}(R)
=\rho(F,\Gamma)[\Gamma: \Lambda],$$
where $R$ is a set of representatives of the quotient $\Lambda^* / \Gamma^*$.
Since this contradicts the assumption that $[\Gamma: \Lambda]$ is not a divisor of $\rho(F,\Lambda)$, 
we must have that $\widehat{f_k} \notin H^{1/2}(\R^d)$ for some $1 \le k \le K$.
\qed \\

\noindent {\em Proof of Theorem \ref{frame-not-riesz}.}
Assume, for the sake of contradiction, that  $\widehat{f_k} \in H^{1/2}(\R^d)$ for all $1\le k \le K$.  
Using Lemma \ref{thm-embedding}, similar reasoning as in the proof of Theorem \ref{sharp-blt}  
implies that each entry of $P=P_{\Lambda^*}(\widehat{F})$ is in $H^{1/2}(\R^d/\Lambda^*).$
The assumption that $K=\rho(F,\Lambda)$, along with
Lemma \ref{lem-constant-rank} and \eqref{frame-condition}, implies that $P$ is full rank almost everywhere. This forces the eigenvalue functions, $\lambda_k$, of $P$ to be nonzero almost everywhere for all $1\le k \le K=\rho(F,\Lambda)$.  However, equation \eqref{frame-condition} shows that the eigenvalue functions are then bounded below by $t^{-1}$ almost everywhere.  This is equivalent to $P$ satisfying the lower bound in equation \eqref{riesz-condition}.  The upper bound also follows from \eqref{frame-condition}.  Thus, $\mathcal{T}^{\Lambda}(F)$ forms a Riesz basis for $V^{\Lambda}(F)$ which gives a contradiction.
\qed

\section{Examples} \label{examples-sec}

The first two examples in this section show that there are multiply generated shift-invariant spaces for which the 
hypotheses of Theorem \ref{sharp-blt} hold, but for which the conclusion of the theorem only holds for a single generator.
The collection of smooth compactly supported functions on $\R^d$ will be denoted by $C_c^{\infty}(\R^d)$.

\begin{example} \label{example-2gen-Rd}
Let $I = [-1/2,1/2)^d$.  Define $f_1 \in L^2(\R^d)$ by $\widehat{f_1}= \chi_I$.  
Take any $g\in C^{\infty}_c(\R^d)$ that is supported on $I$ and satisfies $\|g\|_2=1$, and define $f_2\in L^2(\R^d)$ by
$\widehat{f_2} = g$.

Let $F=\{f_1, f_2\}$, $\Lambda = \mathbb{Z}^d$, and $\Gamma = (\frac{1}{2}\mathbb{Z})\times \mathbb{Z}^{d-1}$.
The space $V^{\Lambda}(F)=V^{\mathbb{Z}^d}(f_1, f_2)$ has the following properties:
\begin{itemize}
\item $\mathcal{T}^{\Lambda}(F)$ is a frame for $V^{\Lambda}(F)$;
\item $V^{\Lambda}(F)$ is $\Gamma$-invariant (it is actually translation invariant);
\item $\widehat{f_1} \notin H^{1/2}(\R^d)$ and $\widehat{f_2} \in C^{\infty}_c(\R^d)\subset  H^{1/2}(\R^d)$;
\item $\rho(F,\Lambda) = 1$ and $[\Gamma : \Lambda] = 2$, so that $[\Gamma : \Lambda]$ does not divide $\rho(F,\Lambda)$.
\end{itemize}
This can be verified by computing the Gramian $P_{\Lambda^*}(\widehat{F})(x)$.  
Note that $\Lambda^* = \Lambda = \mathbb{Z}^d.$
Since $P_{\Lambda^*}(\widehat{F})(x)$ is $\Lambda^*$-periodic, it suffices to only consider $x \in I$ in the subsequent discussion.
A computation shows that for $x \in I$
$$P_{\Lambda^*}(\widehat{F})(x)=
P_{\mathbb{Z}^d}(\widehat{F}) (x)=
\begin{pmatrix}
[\widehat{f_1},\widehat{f_1}]_{\mathbb{Z}^d}(x) & [\widehat{f_1}, \widehat{f_2}]_{\mathbb{Z}^d}(x) \\
[\widehat{f_2},\widehat{f_1}]_{\mathbb{Z}^d}(x) & [\widehat{f_2}, \widehat{f_2}]_{\mathbb{Z}^d}(x) \\
\end{pmatrix}
=
\begin{pmatrix}
1 & \overline{g(x)} \\
g(x) & |g(x)|^2 \\
\end{pmatrix}.
$$
A further computation shows that
$$
\left( P_{\Lambda^*}(\widehat{F})(x) \right)^2 
=  (1+|g(x)|^2)
\begin{pmatrix}
1 & \overline{g(x)} \\
g(x) & |g(x)|^2 \\
\end{pmatrix}
=(1+|g(x)|^2) \ P_{\Lambda^*}(\widehat{F})(x)
$$
Since $g \in C_c^{\infty}(\R^d) \subset L^{\infty}(\R^d)$, we have the operator inequality
$$P_{\Lambda^*}(\widehat{F})(x) \leq \left( P_{\Lambda^*}(\widehat{F})(x) \right)^2  
\leq (1+\|g\|_\infty^2) P_{\Lambda^*}(\widehat{F})(x).$$
So, by \eqref{frame-condition}, $\mathcal{T}^{\Lambda}(F)$ is a frame for $V^{\Lambda}(F)$.

The remaining properties can also be checked easily.  
Similar computations as above, together with Theorem \ref{thm-rank-formula}, show that $V^{\Lambda}(F)$ is $\Gamma$-invariant.
A direct computation shows that $\widehat{f_1} \notin H^{1/2}(\R^d)$.   
The condition $\rho(F,\Lambda)=1$ can be seen by using Proposition \ref{prop-min-generators} 
and noting that $P_{\Lambda^*}(\widehat{F})(x)$ has rank 1 for all $x\in I$.
Finally, it is easily verified that $[\Lambda : \Gamma]=2.$
\end{example}

\begin{example} \label{example-Ngen-R}
Fix any integer $N\geq 2$.  
Let $I=[-1/2,1/2)$ and define $f_{N+1} \in L^2(\R)$ by $\widehat{f_{N+1}} = \chi_I$.  
Fix $0<\epsilon<\frac{1}{2N}$.  Select $f \in C^{\infty}_c(\R)$ with $\|f\|_2=1$ 
such that $f$ is supported on $[0, 1/N]$, and such that $|\widehat{f}(x)|\leq \epsilon$ for all $x \in I$.
For example, such an $f$ can be constructed by suitably dilating and translating a given smooth compactly supported function.
For $1 \leq n \leq N$, define $f_n(x) = f(x- {n}/{N})$.

Define $F=\{f_n\}_{n=1}^{N+1} \subset L^2(\R)$, $\Lambda = \Z$, and $\Gamma = \frac{1}{N} \mathbb{Z}$.
The space $V^{\Lambda}(F)$ satisfies the following properties
\begin{itemize}
\item $\mathcal{T}^{\Lambda}(F)$ is a Riesz basis for $V^{\Lambda}(F)$;
\item $V^{\Lambda}(F)$ is invariant under $\Gamma$;
\item $\rho(F,\Lambda) = N+1$ and $[\Gamma : \Lambda ] = N$, so that $[\Gamma  :  \Lambda ]$ does not divide $\rho(F,\Lambda)$;
\item $\widehat{f_n}\in H^{1/2}(\R)$ for each $1 \leq n \leq N$,
\item $\widehat{f_{N+1}} \notin H^{1/2}(\R)$.
\end{itemize}

The singly generated system $V^{\Z}(f_{N+1})$ is easily seen to be $\frac{1}{N}\Z$-invariant by Theorem \ref{thm-rank-formula} (in fact, $V^{\Lambda}(f_{N+1})$ is translation invariant).  Moreover, the space $V^{\Z}(f_1, \cdots, f_N)$ is $\frac{1}{N}\Z$-invariant by construction.    It follows that $V^{\Lambda}(F)=V^{\Z}(f_1, \cdots, f_{N+1})$ is $\frac{1}{N}\Z$-invariant.  Also, $\widehat{f_n} \in H^{1/2}(\R)$ for each $1 \leq n \leq N$ since $f_n \in C_c^\infty(\R)$.  

Since $\mathcal{T}^{\mathbb{Z}}(\{f_n\}_{n=1}^N)$ is an orthonormal basis for $V^{\mathbb{Z}}(\{f_n\}_{n=1}^N)$, one has
for $1 \leq j, k \leq N$ that $[\widehat{f_j}, \widehat{f_k}]_{\mathbb{Z}}(x) = \delta_{j,k}$ for a.e. $x\in \R$,
for example, see \cite{HSWW}.
Also, if $1 \leq n \leq N$, then for $x\in I$,
$$[\widehat{f_n}, \widehat{f_{N+1}}]_{\mathbb{Z}}(x) = \sum_{j\in \Z} \widehat{f_n}(x-j) \chi_I(x-j)= \widehat{f_n}(x)= e^{-2\pi i n x/N}\widehat{f}(x) .$$
By our assumptions on $f$, we have that for $1 \leq n \leq N$, and $x \in I$,
\begin{equation} \label{eps-period-bnd}
\left| [\widehat{f_{N+1}}, \widehat{f_{n}}]_{\mathbb{Z}}(x) \right| =
\left| [\widehat{f_n}, \widehat{f_{N+1}}]_{\mathbb{Z}}(x) \right| =|\widehat{f}(x)| \leq \epsilon.
\end{equation}

Recalling that $P_{\mathbb{Z}}(\widehat{F}) (x)$ is $\mathbb{Z}$-periodic, we have that for all $x\in I$,
$$P_{\mathbb{Z}}(\widehat{F}) (x) =
\begin{pmatrix}
1 & 0 & \cdots & 0 &  [\widehat{f_1}, \widehat{f_{N+1}}]_{\mathbb{Z}}(x) \\
0 & 1 & \cdots & 0 &  [\widehat{f_2}, \widehat{f_{N+1}}]_{\mathbb{Z}}(x)\\
\vdots   &  \vdots  &  \ddots & \vdots & \vdots \\
0 & 0 & \cdots & 1 & [\widehat{f_N}, \widehat{f_{N+1}}]_{\mathbb{Z}}(x)\\
[\widehat{f_{N+1}}, \widehat{f_{1}}]_{\mathbb{Z}}(x) & [\widehat{f_{N+1}}, \widehat{f_{2}}]_{\mathbb{Z}}(x) & \cdots 
& [\widehat{f_{N+1}}, \widehat{f_{N}}]_{\mathbb{Z}}(x)& 1 \\
\end{pmatrix}.
$$
The Gershgorin circle theorem, together with \eqref{eps-period-bnd}, shows that all eigenvalues of 
$P_{\mathbb{Z}}(\widehat{F}) (x)$ lie in the interval $[1-N\epsilon, 1 + N\epsilon]$.
Since $0<\epsilon<\frac{1}{2N}$, the condition \eqref{riesz-condition} holds with $t=2$, and hence $\mathcal{T}^{\Lambda}(F)$ is a Riesz basis for $V^{\Lambda}(F)$.  Moreover, since $P_{\mathbb{Z}}(\widehat{F}) (x)$ is full rank for a.e. $x\in I$, 
Proposition \ref{prop-min-generators} shows that $\rho(F,\Lambda) = N+1$.
\end{example}

The next example shows that there are multiply generated shift-invariant spaces for which the 
hypotheses of Theorem \ref{frame-not-riesz} hold and for which the conclusion only holds for a single generator.

\begin{example} \label{example2-2gen-Rd}
Let $J=[-1/4,1/4].$  Define $f_1 \in L^2(\R)$ by $\widehat{f_1} = \chi_{J}$.
Select $f_2 \in C^{\infty}_c(\R)$ such that $f_2$ is supported in $[-1/2,1/2]$, $\|f_2\|_2=1$, and $|\widehat{f_2}(x)|<1/2$ for all $x \in J$.

Define $F=\{f_1, f_2\}$ and $\Lambda = \Z$.
The space $V^{\Lambda}(F)$ satisfies the following properties:
\begin{itemize}
\item $\mathcal{T}^{\Lambda}(F)$ is a frame, but not a Riesz basis, for $V^{\Lambda}(F)$;
\item The minimal number of generators $\rho(F,\Lambda)=2$;
\item $\widehat{f_1} \notin H^{1/2}(\R)$ and $\widehat{f_2} \in  H^{1/2}(\R)$.
\end{itemize}
Recall that $P_{\mathbb{Z}}(\widehat{F}) (x)$ is $\mathbb{Z}$-periodic.
A computation shows that for $x\in [-1/2, 1/2]$
$$
P_{\mathbb{Z}}(\widehat{F}) (x)=
\begin{pmatrix}
\chi_J(x) & \chi_J(x) \widehat{f_2}(x) \\
\chi_J(x) \overline{\widehat{f_2}(x)} \ & 1\\
\end{pmatrix}.
$$
For $\ 1/4 < |x| < 1/2$, we have 
$$ P_{\mathbb{Z}}(\widehat{F}) (x)=
\begin{pmatrix}
0 & 0 \\
0 & 1 \\
\end{pmatrix},
$$
so that $\lambda_1(x)=1$ and $\lambda_2(x)=0$, and for $|x| < {1}/{4}$, we have
$$ P_{\mathbb{Z}}(\widehat{F}) (x)=
\begin{pmatrix}
1 & \widehat{f_2}(x) \\
\overline{\widehat{f_2}(x)} & 1 \\
\end{pmatrix}
$$
so that $\lambda_1(x)=1+|\widehat{f_2}(x)|$ and $\lambda_2(x)=1-|\widehat{f_2}(x)|$.  

By \eqref{riesz-condition}, $\mathcal{T}^{\Z}(F)$ is not a Riesz basis for $V^{\Z}(F)$.  However Proposition \ref{prop-min-generators}, \eqref{frame-condition}, and $|\widehat{f_2}(x)|<1/2$ for $x\in J$, show that $\mathcal{T}^{\Z}(F)$ is a frame for $V^{\Z}(F)$ and $\rho(F,\Z)=2$.  
\end{example}



\section*{Acknowledgments}
The authors thank Jeff Hogan, Joe Lakey and Qiyu Sun for helpful discussions related to the literature on shift-invariant invariant spaces. The authors also thank Charly Gr\"ochenig for useful comments related to the Balian-Low theorem.

D.~Hardin was partially supported by NSF DMS Grant 1109266, NSF DMS Grant 0934630, NSF DMS Grant 1521749, NSF DMS Grant 1516400, and
NSF DMS Grant 1412428.
M.~Northington was partially supported by NSF DMS Grant 0934630, NSF DMS Grant 1521749, and NSF DMS Grant 1211687.
A.~Powell was partially supported by NSF DMS Grant 1211687 and NSF DMS Grant 1521749.
A.~Powell gratefully acknowledges the hospitality and support of the Academia Sinica Institute of Mathematics (Taipei, Taiwan).

\section*{Appendix}

In this Appendix, we include a proof of the characterization of Sobolev spaces on the torus from Lemma \ref{prop-equiv-norm}.\\

\noindent{\em Proof of Lemma \ref{prop-equiv-norm}.}

{\em Step I.} We first prove \eqref{eq-sob-norm1}.  
By Parseval's theorem
\begin{align}
|M_\Gamma|\int_{M_{\Gamma}} \int_{{M}_\Gamma} \frac{|f(x+y)-f(x)|^2}{|y|^{d+2s}} dx dy  
& = \int_{M_{\Gamma}}\frac{ \sum_{\xi \in \Gamma^*} |\widehat{f}(\xi)|^2 |e^{-2\pi i \xi \cdot y} -1|^2}{|y|^{d+2s}} dy \notag \\
&= \sum_{\xi \in \Gamma^*} |\widehat{f}(\xi)|^2 G(\xi), \label{eq:Gkterm}
\end{align}
where $G(\xi)= \int_{M_{\Gamma}} \frac{ |e^{-2\pi i \xi \cdot y} -1|^2}{|y|^{d+2s}} dy.$
For $\xi \neq 0$, let $\xi'={\xi}/{|\xi|}$ and note that
\begin{align}
G(\xi)
= 2\int_{M_{\Gamma}} \frac{ 1-\cos(2\pi \xi \cdot y)}{|y|^{d+2s}} dy=2 |\xi|^{2s}\int_{|\xi|M_{\Gamma}} \frac{ 1-\cos(2\pi \xi '\cdot y)}{|y|^{d+2s}} dy. \label{G-eq1}
\end{align}

Since $\Gamma$ and $\Gamma^*$ are full-rank lattices, there exists $C>0$ such that $|\xi| \geq C>0$ for all $\xi \in \Gamma^* \backslash \{ 0 \}$.  Moreover, there exists $r>0$ such that
$B(0,r) = \{x \in \R^d : |x|<r\} \subset {M}_{\Gamma}$.  
Thus,
\begin{equation} \label{ball-M-eq}
\forall \xi \in \Gamma^* \backslash \{0\}, \ \ \ B(0,Cr) \subset |\xi|{M}_{\Gamma}.
\end{equation}
Combining \eqref{G-eq1} and \eqref{ball-M-eq} shows that $G(\xi)$ satisfies the lower bound
\begin{equation} \label{G-low-bnd}
\forall \xi \in \Gamma^* \backslash \{0 \}, \ \ \ 2|\xi|^{2s} \int_{B(0,Cr)}\frac{ 1-\cos(2\pi \xi'\cdot y)}{|y|^{d+2s}} dy \le G(\xi).
\end{equation}
Moreover \eqref{G-eq1} gives the upper bound
\begin{equation} \label{G-up-bnd}
G(\xi) \leq  2|\xi|^{2s} \int_{\R^d} \frac{ 1-\cos(2\pi \xi'\cdot y)}{|y|^{d+2s}} dy.
\end{equation}
Notice that the integrals in \eqref{G-low-bnd} and \eqref{G-up-bnd} are invariant under $y \mapsto Uy$ for all unitary $U: \R^d \to \R^d$.  It follows that 
these integrals are independent of the actual value of $\xi'$.  The integral in \eqref{G-low-bnd} is positive since the integrand is a.e. positive.  It is also easily seen that the integral in \eqref{G-up-bnd} is finite.  Thus, $$G(\xi) \asymp  |\xi|^{2s}.$$
This, together with \eqref{eq:Gkterm} completes the proof of \eqref{eq-sob-norm1}.\\

{\em Step II.} Next, we prove \eqref{eq-sob-norm2}.  
By Parseval's theorem
\begin{align}
|M_\Gamma|\sum_{j=1}^d\int_{[-\frac{1}{2},\frac{1}{2})} \int_{{M}_\Gamma} \frac{ |f(x+ta_j)-f(x)|^2}{|t|^{1+2s}} dx dt  
&= \sum_{j=1}^d\int_{[-\frac{1}{2},\frac{1}{2})}\frac{\sum_{\xi \in \Gamma^*} |\widehat{f}(\xi)|^2 |e^{-2\pi i \xi \cdot t a_j}-1|^2}{|t|^{1+2s}}   dt \notag \\
	&=\sum_{\xi \in \Gamma^*} |\widehat{f}(\xi)|^2 H(\xi), \label{part2-eq1}
\end{align}
where 
$$H(\xi)=\sum_{j=1}^d\int_{[-\frac{1}{2},\frac{1}{2})}\frac{|e^{-2\pi i t \xi \cdot a_j}-1|^2}{|t|^{1+2s}}dt.$$

For $\xi \in \Gamma^* \backslash \{0\}$, let $\xi'={\xi}/{|\xi|}$, and note that
\begin{align} 
H(\xi) &= 2\sum_{j=1}^d \int_{[-\frac{1}{2},\frac{1}{2})} \frac{ 1-\cos(2\pi t \xi \cdot a_j)}{|t|^{1+2s}} dt \notag \\
&=2 |\xi|^{2s}   \sum_{j=1}^d  \int_{-{|\xi|}/{2}}^{{|\xi|}/{2}} \frac{ 1-\cos(2\pi t \xi'\cdot a_j)}{|t|^{1+2s}} dt \label{Heq1}\\
&\leq 2d |\xi|^{2s}   \int_{-\infty}^{\infty} \frac{1}{|t|^{1+2s}} dt. \label{Hupbnd}
\end{align} 

Let $A$ be the $d\times d$ matrix with $a_j$ as its $j$th column.  Then $a_j = Ae_j$ where $\{e_j\}_{j=1}^d \subset \mathbb{Z}^d$ is the canonical
basis for $\R^d$.  Since $\{a_j\}_{j=1}^d$ is a basis for $\Gamma$, the matrix $A$ is invertible.  If $\sigma_d \geq \sigma_1>0$ are the respective largest
and smallest singular values of $A$ (and hence also $A^*$) then
$$\forall x \in \R^d, \ \ \ \sigma_1 |x| \leq | A^*x | \leq \sigma_d |x|.$$
Moreover, if $\xi \in \Gamma \backslash \{ 0 \}$, then $\xi = (A^*)^{-1}z$ for some $z\in \mathbb{Z}^d \backslash \{ 0 \}$,
so that
$$\forall \xi \in \Gamma^* \backslash \{ 0 \}, \ \ \ \sigma_d^{-1} \leq \sigma_d^{-1}|z| \leq |(A^*)^{-1} z| = |\xi|$$
Let $\rho =1/(4\sigma_d).$  
Using that $(1 - \cos(2\pi \theta) )\geq {\theta^2}/{2}$ for all $|\theta| \leq 1/4$, it follows from \eqref{Heq1} that for all $\xi \in \Gamma^* \backslash \{ 0\}$

\begin{align} 
H(\xi) &=2 |\xi|^{2s}   \sum_{j=1}^d  \int_{-{|\xi|}/{2}}^{{|\xi|}/{2}} \frac{ 1-\cos(2\pi t A^*(\xi')\cdot e_j)}{|t|^{1+2s}} dt \notag \\
& \geq 2 |\xi|^{2s}   \sum_{j=1}^d  \int_{-\rho}^{\rho} \frac{ 1-\cos(2\pi t A^*(\xi')\cdot e_j)}{|t|^{1+2s}} dt \notag \\
& \geq 2 |\xi|^{2s}   \sum_{j=1}^d  \int_{-\rho}^{\rho} \frac{|2\pi t A^*(\xi')\cdot e_j|^2}{2|t|^{1+2s}} dt \notag \\
&=4\pi^2 |\xi|^{2s} \left( \sum_{j=1}^d |A^*(\xi')\cdot e_j|^2 \right)  \int_{-\rho}^{\rho} \frac{1}{|t|^{2s-1}} dt \notag \\
& = 4\pi^2 |\xi|^{2s} |A^*(\xi')|^2   \int_{-\rho}^{\rho} \frac{1}{|t|^{2s-1}} dt \notag \\
& \geq 4\pi^2 \sigma_1 |\xi|^{2s}  \int_{-\rho}^{\rho} \frac{1}{|t|^{2s-1}} dt . \label{Hlowbnd}
\end{align}

Hence, by \eqref{Hupbnd} and \eqref{Hlowbnd}
$$H(\xi) \asymp |\xi|^{2s}.$$
This, together with \eqref{part2-eq1} completes the proof of \eqref{eq-sob-norm2}.
\qed

\end{document}